\documentclass[12pt,fleqn]{amsart}

\usepackage{color}
\usepackage{enumerate}
\usepackage{amssymb}
\usepackage{amsmath}

\newcommand{\cA}{\mathcal A}
\newcommand{\fA}{\mathfrak A}
\newcommand{\cB}{\mathcal B}

\newcommand{\cF}{\mathcal F}

\newcommand{\sub}{\subseteq}

\newcommand{\vf}{\varphi}
\newcommand{\eps}{\varepsilon}
\newcommand{\MA}{\mbox{\sf MA}}

\newcommand{\JL}{\protect{\rm JL}}
\newcommand{\To}{\longrightarrow}

\usepackage{graphicx} 
\usepackage{amsfonts}
\usepackage{amsthm}
\usepackage{enumerate}
\usepackage{tikz-cd}
\usepackage{amscd}
\usepackage{hyperref}
\usepackage{mathrsfs}
\usepackage{xcolor}
\usepackage{setspace}

\newtheorem{theorem}{Theorem}[section]

\newtheorem{corollary}[theorem]{Corollary}
\newtheorem{lemma}[theorem]{Lemma}

\newtheorem{problem}[theorem]{Problem}

\theoremstyle{definition}
\newtheorem{definition}[theorem]{Definition}

\newtheorem{remark}[theorem]{Remark}

\newcommand{\sfrown}{^{\scalebox{0.7}[0.7]{$\frown$}}}
\newcommand{\lev}{\protect{\rm Lev}}
\newcommand{\sm}{\setminus}
\newcommand{\odeg}{\protect{\rm odeg}}
\newcommand{\cV}{\mathcal V}

\newcommand{\norm}[1]{\left\lVert#1\right\rVert}
\newcommand{\abs}[1]{\left\lvert#1\right\rvert}

\newcommand{\suc}{\protect{\rm suc}}

\title[AD families and JL-spaces]{On almost disjoint families\\ and Johnson-Lindenstrauss Spaces}
\author[A. Avil\'{e}s]{Antonio  Avil\'{e}s}
\address{Departamento de Matem\'{a}ticas\\
Facultad de Matem\'{a}ticas\\ Universidad de Murcia\\ 30100 Espinardo, Murcia\\
Spain} \email{avileslo@um.es}
\author[G.\ Plebanek]{Grzegorz  Plebanek}
\address{Instytut Matematyczny, Uniwersytet Wroc\l awski}
\email{grzegorz.plebanek@math.uni.wroc.pl}

\author[L.\  Sáenz]{Luis Sáenz}
\address{Centro de Ciencias Matem\'aticas, UNAM
}
\email{luisdavidr@ciencias.unam.mx}
\date{\today}


\thanks{Authors supported by Fundaci\'{o}n S\'{e}neca - ACyT Regi\'{o}n de Murcia project 21955/PI/22, Avil\'{e}s supported by ERDF and MICIU/AEI/ 10.13039/501100011033 project PID2021-122126NB-C32.
Sáenz supported by SECIHTI grant CBF-2025-I-898.}

\begin{document}

\begin{abstract}
Every almost disjoint family $\cA$ of infinite subsets of $\omega$ gives rise to a scattered compact
space $K_\cA$ and a family of Johnson-Lindenstrauss spaces $\JL_p(\cA)$.
We consider almost disjoint families arising from  finitely branching trees and investigate whether they
define homeomorphic compacta and isomorphic Banach spaces. 
\end{abstract}
\maketitle

\section{Introduction}

A family $\mathcal{A}$ of infinite subsets of a countable set $N$ is almost disjoint if the intersection of every two different sets in $\mathcal{A}$ is always finite. Associated to such a family, we have the following objects:

\begin{itemize}
	\item The topological space $\psi(A) = N \cup \mathcal{A}$ where points in $N$ are isolated, and the neighborhoods of $A\in \mathcal{A}$ are the sets $V$ such that $A\in V$ and $N\cap V\setminus A$ is finite.

	\item The compact space $K_\mathcal{A} = \psi(A)\cup\{\infty\}$, the one-point compactification of $\psi(A)$.
	
	\item The Banach space of continuous functions $C(K_\mathcal{A})$, also known as the Johnson-Lindenstrauss space $\JL(\cA)$ or $\JL_\infty(\cA)$.
	
	\item The Johnson-Lindenstrauss space $\JL_p(\cA)$, for $p\in (1,\infty]$, a Banach space obtained from $\JL(\cA)$ by considering a new norm in which $\ell_p$-sums of values at points $A\in\cA$ are also taken into account. 
	(This construction can also be applied with $p=1$ but then the resulting space is, in a sense, trivial;  see section \ref{3}.)
	\end{itemize}

The spaces $\psi(A)$ and $K_\mathcal{A}$, first considered by Alexandroff and Urysohn and sometimes called Isbell-Mr\'{o}wka spaces, are well known animals in the fauna of general topology, cf. \cite{HHHsurvey, Hr14}. The spaces $\JL_p(\cA)$ were introduced in 
\cite{JLoriginal}, see also \cite{YostOriginal}, and have also been studied as a rich source of counter-examples \cite{FKK13,KL21,Avilésarticle,GHK23}. They played  a crucial role in proving  the main result of  \cite{AMP20} (on twisted sums of $c_0$)
and in \cite{PS23} which gives  the first construction of a complemented subspace of a $C(K)$-space
which is not a $C(K)$-space itself. For many purposes, the choice of the family $\cA$ is not relevant, but it is also often the case that $K_\mathcal{A}$ and $\JL_p(\cA)$ will display different behaviors, and serve as counterexamples to different conjectures, depending on the properties of the underlying family $\cA$. There are, indeed, $2^\mathfrak{c}$ many non-isomorphic Banach spaces of the form $\JL(\cA)$ \cite{CABELLOSANCHEZ2020108571}. The way in which this variety usually appears in the literature is either by general counting arguments, or by the construction of families with peculiar combinatorial properties using set-theoretic techniques. 

In this paper, we turn our attention to the most concrete almost disjoint families that one can naturally define without recurring to transfinite methods, and we wonder how different their spaces $K_\mathcal{A}$ and $\JL_p(\cA)$ could be. The epitome of such a family is the set $\cA_2$ of the branches of the dyadic tree $2^{<\omega}$. Marciszewski \cite{Ma89} and Marciszewski and Pol \cite{MP09} showed that one can obtain many different spaces $K_\mathcal{A}$ by restricting to sufficiently different subsets $\mathcal{A}\subset \mathcal{A}_2$. But their results do not give any insight if we look at something so similar to $\mathcal{A}_2$ as the set $\mathcal{A}_n$ of the branches of the $n$-adic tree $n^{<\omega}$, for $n>2$. We prove that if $n>m$, then  $K_{\cA_n}$ is not homeomorphic to (not even a continuous image of) $K_{\cA_m}$, but nevertheless $\JL_p(\cA_n)$ and $\JL_p(\cA_m)$ are isomorphic Banach spaces. 

Marciszewski and Pol identified a certain property of Banach spaces that allows one to distinguish between some $\JL$-spaces. 
If $\mathcal B$ is the almost disjoint family associated with the full Baire tree $\omega^{<\omega}$ and $\mathcal A$ arises from a finitely branching tree, then the Banach spaces $\JL(\mathcal B)$ and $\JL(\mathcal A)$ are not isomorphic; see \cite[Corollary~3.3 and Theorem~3.4]{MP09}. 
As a consequence, $\JL_2(\mathcal B)$ and $\JL_2(\mathcal A)$ are also not isomorphic; see \cite[Theorem~6.1]{MP09}.
We were able to distinguish certain finitely branching trees from this perspective:
Our result states  that if $\mathcal B$ is the family of branches of a finitely branching tree~$\Upsilon$ in which the number of immediate successors of each node increases with its height, then $\JL_p(\mathcal B)$ is not isomorphic to $\JL_p(\mathcal A_2)$ for  any $p>1$.

Different phenomena occur when looking at almost disjoint families of size less than $\mathfrak{c}$. On the one hand, there is a variety of them. For example, Barriga-Acosta and Hern\'{a}ndez-Hern\'{a}ndez \cite{BARRIGAACOSTA20191} proved that for a cardinal number of uncountable cofinality, there are $\mathfrak{c}$ almost disjoint families of size $\kappa$, all of which are homeomorphic as subsets of $2^\omega$, yet whose associated compact spaces $K_\cA$ are pairwise non-homeomorphic. However, at the Banach space level uniformity appears. Koszmider and Rojek~\cite{KR25} proved that, under Martin’s axiom, for every pair of almost disjoint families of the same cardinality less than~$\mathfrak c$, there exists an automorphism of $\ell_\infty/c_0$ mapping one family onto the other. And it was shown in \cite{CABELLOSANCHEZ2020108571} that if Martin's axiom $\MA(\kappa)$ holds for a cardinal $\kappa$, then all spaces $\JL_\infty(\cA)$ with $|\mathcal{A}|=\kappa$ are isomorphic. The proof of this latter result was built on  a highly nontrivial result on twisted sums of $c_0$ and $C(K)$-spaces.
We give a short elementary proof that, for any $p\in (1,+\infty]$, all spaces $\JL_p(\mathcal{A})$ are isomorphic for families $\mathcal{A}$ of size $\kappa$ under $\MA(\kappa)$.

The paper is organized as follows. Section~\ref{2} deals with the topological classification of spaces $K_\cA$ when $\cA$ is the family of branches of a tree. Sections~\ref{4} and  \ref{5}  contain several results on
 the isomorphic classification of the corresponding spaces$\JL_p(\cA)$. 
Section~\ref{6} deals with almost disjoint families of size less than $\mathfrak{c}$. Besides the facts mentioned in this introduction, other finer but more technical results are included along the text. Finally, Section~\ref{sectionproblems} lists some open problems.

The main ideas used in this article were found by the authors. We used AI (Chatgpt 5.5) to revise and improve earlier versions of the work.

\section{Almost disjoint  families and compact spaces}\label{2}

We say that $\mathcal A$ is an almost disjoint family (AD family) if $\cA$ is a family of infinite subsets of some countable set
and $A\cap B$ is finite for any pair of distinct $A,B\in\cA$. Given an AD family $\cA$ on a countable set $N$, we define the $\psi$-space associated with it as the set 
$\psi(\mathcal A)=N \cup \mathcal A$ where
 $N$ is a discrete subspace, and  for  every $A\in \mathcal A$ the sets 
 \[ (A\setminus F) \cup \{A\},  \mbox{ where } F\in [N]^{<\omega},\]
 form a local base at the point $A$ of the space $\psi(\cA)$.
 It is easy to check that $\psi(\cA)$ is  
 a locally compact, separable, zero-dimensional Hausdorff space.

We  denote by  $K_{\mathcal A}$ the Alexandrov-Urysohn compactum associated  with $\mathcal A$, that is,  the one point compactification of $\psi(\mathcal A)$. Such a space is separable and scattered of height 3.
Note  that $K_\cA$ may also be  described as follows. Let $\fA$ be the algebra of subsets of $N$ generated by
$\cA$ and all finite sets. Then $K_\cA$ is (homeomorphic to) the Stone space of ultrafilters on the algebra $\fA$.

A tree $T$ is a partially ordered set with a minimum element such that every initial segment $\{s\in T : s\prec t\}$ is well ordered. A subtree of $T$ is a downwards closed subset $S\subseteq T$. We will focus on subtrees $T$ of the full tree $\omega^{<\omega}$, the set of finite sequences of natural numbers, ordered by end-extension.
The symbol $[T]$ denotes the set of all $x\in\omega^\omega$
for which every initial segment $x|n$ belongs to $T$. We write $\tau^\frown \sigma$ for the concatenation of $\tau$ and $\sigma$.
For $\tau\in T$ and $x\in \omega^\omega$ we write $\tau\prec x$ if $\tau$ is an initial segment of $x$, and 
\[ [\tau]=\{x\in [T]: \tau \prec x\}.\]
Recall that the family of all $[\tau]$ for $\tau\in T$  is  a base for the topology on $[T]$ inherited from 
the Baire space $\omega^\omega$.

Given a tree $T$ and $x\in [T]$,  we write $B(x)=\{x|n:n\in\omega\}$ for
the  branch corresponding to $x$. Note that $x$ also might be called  a branch.
However, in our setting the distinction between $x$ and $B(x)$ will
be convenient.
The sets $B(x)$ form an almost disjoint family on $T$ --- we write $K_T$
for the compact space associated to such a family. Thus
$K_T$ has all the nodes of $T$ as a set of isolated points and all $x\in [T]$ at the first level.
Note that the sets   $\{x\}\cup B(x)\sm F$ where $F$ is finite form a local base at $x\in K_T$.

In fact, for any subset $X$ of the Baire space $\omega^\omega$ we can  consider the almost disjoint family
\[ \cA_X=\{B(x): x\in X\}.\]
Note that, formally speaking, $\cA_X$ may be treated as an AD family on $\omega^{<\omega}$ or on its subset
$\{x|n: x\in X, n\in\omega\}$. This distinction, however, is inessential from our point of view.

\begin{remark}\label{rem1}
Suppose that $\cA$ is an AD family on $N$ and $N'\supseteq N$ is a larger  countable set.
If $\cA$ is not maximal then the $\psi$-spaces with base sets $N\cup\cA$ and $N'\cup\cA$ are homeomorphic.

Indeed, there is an infinite set $B\sub N$ which is almost disjoint from every $A\in\cA$. If $N'\setminus N$ is finite, any bijection will do. 
Suppose now that $N'\sm N$ is infinite; 
divide $B$ into two infinite parts $B_0, B_1$, and further divide $B_1$ into two infinite parts $B_{1,0},B_{1,1}$.
Then one can define the desired homeomorphism by
mapping  $N'\sm N$ injectively onto $B_0$, $B_0$ onto $B_{1,0}$, and $B_1$ onto $B_{1,1}$. 
\end{remark}

Marciszewski and Pol \cite[section 7]{MP12}  proved the following result.

\begin{theorem} \label{2:2}
There is a family $\cF$ of size continuum of closed subsets of the Cantor set such that  
whenever $F\neq H$ are in $\cF$  then the spaces $K_{\cA_{F}}$ and $K_{\cA_{H}}$ are not homeomorphic.
\end{theorem}

The core of the argument used to prove \ref{2:5} and \ref{2:6} can be informally stated as follows: If subtrees 
$T,S$ of $\omega^{<\omega}$ are substantially different then the compacta $K_{T}$ and $K_{S}$
are not homeomorphic. This observation was first made  in the unpublished dissertation of Poprawa \cite{Po22},
building on ideas from \cite{MP12}. We expand it further in this section, 
presenting two criteria expressing that trees differ to such an extent
that they give rise to incomparable compact spaces.

Below we denote by $|\tau|$ the length of $\tau\in T$, and
\[\lev_n(T)=\{\tau\in T: |\tau|=n\};\]
\[ \lev_n(T,\tau)=\{t\in\lev_n(T): \tau\preceq t\}; \]
\[ \suc_T(\tau)=\{\tau\sfrown n\in T: n\in\omega\}.\]

\begin{theorem}\label{2:5}
Let $T$ and $S$ be   finitely branching subtrees of $ \omega^{<\omega}$. Suppose that 
for every $\tau \in T$ and $k\in\omega$ there is $n\in\omega$ such that 
\[ \abs{\lev_n(T,\tau) } > \abs{\lev_{k+n}(S)}.\]
Then the compact spaces $K_T$ and $K_{S}$ are not homeomorphic.
\end{theorem}

\begin{proof}
 We argue by contradiction: suppose that $f:K_{S}\to K_T$ is a homeomorphism. The points of $T$ are isolated in the space  $K_T$ and $\infty$ is the only non-isolated point of $K_T$ that does not have a neighborhood that is a convergent sequence. The same applies for $S$, so $f(S)=T$ and $f(\infty)=\infty$.
   
For every branch $x\in [T]$ there is a natural number $k_x$ such that if we take $y\in [S]$ such that
$f(y)=x$ then $f(\tau)\prec x$ for every $\tau\prec y$ of length $\ge k_x$. Hence we can decompose $[T]$
into countably many pieces
\[ [T]=\bigcup_k X_k;\quad \mbox{ where } X_k=\{x\in [T]:  k_x= k\}.\]

By the Baire category theorem, there is $k$ and $\tau\in T$ such that $X_k\cap [\tau]$ is
dense in $[\tau]$. We now consider $n$ whose  existence is given by the assumption of the theorem.
For every $t \in \lev_{n}(T,\tau)$ pick $x_t\in X_k$ such that $t\prec x_t$, and   $y_t\in [S]$
satisfying
$f(y_t)=x_t$. Define 
\[ \vf:\lev_n(T,\tau)\to \lev_{k+n}(S) \mbox{ by } \vf(t)=y_t|(k+n).\]
 To obtain a contradiction, we show that   $\vf$ is injective.

Take $t,t'\in \lev_n(T,\tau)$ for which  $\vf(t)=\vf(t')$. Consider $\sigma_j=\vf(t)|j$ for $j=k, \ldots, k+n$. We have $n+1$ elements,
 we know that $f(\sigma_j)\prec x_t$ and $f(\sigma_j)$'s  are pairwise different. Hence  there is $j$ such that
the length of $f(\sigma_j)$ is at least $n$. At the same time $f(\sigma_j)\prec x_{t'}$ so 
\[ t'=x_{t'}|n= f(\sigma_j)|n = x_t|n=t,\]
and the proof is complete.
\end{proof}

The next result relies on  a general observation concerning continuous functions on scattered
spaces. Here $K'$ and $K''$ denote the first and the second Cantor-Bendixson derivative of a space $K$.

\begin{lemma}\label{2:6l}
Let $K$ and $L$ be scattered compact spaces such that $K''=\{a\}$ and $L''=\{b\}$.
If $g:K\to L $ is a continuous surjection then for every $x\in L'\sm\{b\} $ the set
$g^{-1}(x)\cap K'$ is nonempty and finite.
\end{lemma}

\begin{proof}
Recall first that $L'\sub g[K']$. Indeed,  suppose that $x\in L'\sm g[K']$ and
take a clopen set $V\ni x$ which is disjoint from $g[K']$. Then $V$ is  infinite while $g^{-1}[V]\sub K\sm K'$ is a compact set of isolated points,
which is impossible.

By the same argument $g[K'']$ contains $L''$, that is $g(a)=b$. Therefore, 
$g^{-1}(x)\cap K'\neq\emptyset$ for every $b\neq x\in L'$ and
this set  must be finite as a closed subset of the discrete set $K'\sm\{a\}$.
\end{proof}

\begin{theorem}\label{2:6}
Let $T,S\subseteq \omega^{<\omega}$ be finitely branching trees and suppose that 
for every $\tau \in T$ 
there exists $t\in T$ with $\tau \preceq t$ such that 
\[ \abs{\suc_T(t)} > \abs{\suc_S(\sigma)} \mbox{ for all } \sigma\in S.\]
Then  there is no continuous surjection $K_S\to K_{T}$.
\end{theorem}

\begin{proof}
 Suppose, towards a contradiction, that  $f:K_{S}\to K_T$ is a continuous, onto map.
By Lemma \ref{2:6l}, $f(\infty)=\infty$ and the set $\{y\in [S]: f(y)=x\}$ is finite and nonempty for every $x\in [T]$. This means that $f^{-1}(x\cup \{s: s\prec x\})$ is a clopen subset of $K_S$ that has a finite nonzero number of elements in the middle level formed by the branches. Such a clopen set has to be the union of finitely many branches modulo finitely many isolated points.
Hence, to every $x\in [T]$ we can associate $\sigma_x\in S$
so that

\begin{enumerate}[(i)]
\item there is exactly one $y_x\in [\sigma_x]$ such that $f(y_x)=x$;
\item for every $\sigma\in S$ with $\sigma_x\preceq\sigma$  we have $f(\sigma)\preceq x$ if and only if  $\sigma\preceq y_x$.
\end{enumerate}


In this way,  to every $x\in [T]$ we have assigned a parameter from a countable set and the Baire category argument implies that
there is $\tau\in T$ such that the set of those $x\in[\tau]$ for which $\sigma_x$ is equal to a fixed $\sigma$  
is  a  dense subset of  $[\tau]$.

Take $t\in T$ such that $\tau\preceq t$, 
 $\suc_T(t)=\{t_0,\ldots, t_{m-1}\}$
and $m$  is larger than the cardinality of successors in $S$.

It follows that we have $x_0,\ldots, x_{m-1}\in [T]$ such that $x_i\in [t_i]$ for $i<m$, 
all sharing the same $\sigma$. Then for every
$i<m$ there is the unique $y_i\in[\sigma]$ with $f(y_i)=x_i$.

Now the point is that every pair $y_i, y_j$ splits somewhere above  but they all cannot split at the same level.
Hence, there is $\sigma_1$ with $\sigma\preceq\sigma_1$ and there are $i,j,k<m$ 
such that $y_i,y_j\in[\sigma_1]$ while $y_k\notin [\sigma_1]$. 
Then, by $(ii)$ above,
$f(\sigma_1)\preceq x_i$ and $f(\sigma_1)\preceq x_j$ and hence  $f(\sigma_1)\preceq t$.
On the other hand, 
$f(\sigma_1)\not\preceq x_k$, which is  impossible.
\end{proof}

\begin{corollary}\label{2:7}
For each $m\ge 2$, let  $K_{m}$ denote  the compact space associated to the branches of the $m$-adic tree
$m^{<\omega}$.
Then the spaces $K_m$  are pairwise non-homeomorphic. Furthermore, if $n<m$ then there is no continuous surjection from 
$K_{n}$ onto  $K_{m}$.
\end{corollary}

\begin{proof}
This follows directly from Theorem \ref{2:6}.
\end{proof}

Let us remark that Theorem \ref{2:5} could also be applied for $T=m^{<\omega}$  and $S=p^{<\omega}$ when $p<m$, but it gives a weaker conclusion. Indeed,  for every $\tau\in T$ of length $n_0$, we have that $|\lev_n(T,\tau)|=m^{n-n_0}$
while $ |\lev_{k+n}(S)|=p^{k+n}< m^{n-n_0}$ for large $n$.

\section{Johnson-Lindenstrauss spaces}\label{3}

For any compact space $K$ we denote by $C(K)$ the Banach space of continuous real-valued functions
on $K$, equipped with the usual supremum norm. As always, we identify the dual space
$C(K)^\ast$ with $M(K)$ --- the space of signed Radon measures on $K$ having finite variation.
We sometimes write $\mu(f)$ rather than $\int_K f\; {\rm d}\mu$.
Given $\mu\in M(K)$, the norm of $\mu$ is equal to $|\mu|(K)$, where $|\mu|$ is the variation of $\mu$.
Recall that if $K$ is scattered then every $\mu\in M(K)$ is purely atomic.
For any $x\in K$ we denote by $\delta_x$ the corresponding Dirac measure.

Let $\cA$ be an almost disjoint family on $\omega$. For any $p>1$ and $f\in C(K_\cA)$ we let
\[ \norm {f}'_p=\Big(\sum_{A\in \mathcal A}\abs{f(A)}^p\Big)^{1/p} \mbox{ and } \norm f_p=\max\{\norm {f}_\infty,\norm{f}'_p\}.\]

\begin{definition}\label{3:0}
   The $p$-Johnson-Lindenstrauss space associated with $\cA$ is the space
\[\JL_{p}(\mathcal A)=\{f\in C(K_\mathcal A): \norm{f}_p'<\infty\},\]
equipped with the norm $\norm \cdot_p$.
\end{definition}

Clearly, if $1<p<\infty$ and $f\in \JL_p(\cA)$ then necessarily $f(\infty)=0$. By analogy,  the space $C(K_\cA)$ may also be 
denoted by $\JL_\infty(\cA)$.
$\JL$-spaces  may be seen as twisted sums of some familiar Banach spaces.
Concerning the terminology that we use here, we refer the reader to a  recent monograph \cite{CSC}
by Cabello Sánchez and Castillo.

Recall that
an \emph{exact sequence}  of Banach spaces is a diagram
\[
\begin{CD} 0@>>> Y @>\jmath>> Z @>\rho>>
X@>>>0
\end{CD}\]
formed by Banach spaces and linear continuous operators in which the kernel of each arrow coincides with the image of the
preceding one. The middle space $Z$ is usually called a \emph{twisted sum} of $Y$
and $X$. By the open mapping theorem, $Y$ must be isomorphic to a subspace of $Z$ and $X$ to the quotient $Z/Y$.

If $\cA$ is an almost disjoint family and $|\cA|=\kappa$ then (for every $p> 1$)
\[
\begin{CD} 
0@>>> c_0 @>>> \JL_\infty(\cA) @>>>c_0(\kappa)@>>>0,\\
0@>>> c_0 @>>> \JL_p(\cA) @>>>\ell_p(\kappa)@>>>0, 
\end{CD}\]
see \cite[page 102]{CSC}.

Definition \ref{3:0} might be formally applied for $p=1$. This case, however, becomes trivial as it is not difficult to check that
\[ \JL_1(\cA)\simeq c_0\oplus \ell_1(|\cA|).\]

Let us record the following simple observations. For any AD families $\cA, \cA'$ and $p\in (1,\infty]$,
if $\JL_p(\cA)\simeq \JL_p(\cA')$ then $|\cA|=|\cA'|$. Moreover,
\[ \JL_p(\cA)\simeq c_0\oplus \JL_p(\cA).\]

Marciszewski and Pol \cite{MP09} proved that the Banach spaces 
$C(K_{2^{<\omega}})$ and $C(K_{\omega^{<\omega}})$ are not isomorphic.
Earlier, Marciszewski \cite{Ma89} introduced  an index (of Rosenthal compact spaces) that enables one to
distinguish Banach spaces of the form $C(K_{\cA_X})$ for Borel sets $X\sub\omega^\omega$.
These results are also outlined in \cite[section 6 and 7]{CABELLOSANCHEZ2020108571}.
The conclusion is that there is a family $\{X_\xi:\xi<\omega_1\}$ of Borel subsets of $\omega^\omega$ such that,
\[ C(K_{\cA_{X_\xi}})\not\simeq C(K_{\cA_{X_\eta}})\mbox{ whenever } \xi\neq\eta.\]
In the next section we present a couple of results on isomorphisms between various $\JL$-spaces. 
In the rest of this section we record some auxiliary results.

Given a pair of compact spaces $K\sub L$, by an \textit{extension operator} 
$E:C(K)\to C(L)$ we mean a bounded linear operator (between Banach spaces of continuous functions) 
such that $E(f|_K) = f$ for every $f\in C(K)$.
Recall that if such an operator exists then $C(K)$ is isomorphic to a complemented 
subspace $X=\{E(f): f\in C(K)\}$ of $C(L)$. Indeed, we have then a projection $P:C(L)\to X$
defined as $P(g)=E(g|_K)$, $g\in C(L)$.
Note that if there is a continuous retraction $r$ from $L$ onto $K$ then
$C(K)\ni f\mapsto f\circ r \in C(L)$ gives an extension operator of norm one.

We shall consider such a pair of compact spaces $K,L$ where $L$ is a countable discrete
extension of $K$, that is  $L=K\cup N$ and $N$ is a 
countable infinite set of isolated points (and $N\cap K=\emptyset$).
 In this setting we recall the following standard fact, see \cite[Lemma 2.7]{MP18}.
 
 \begin{lemma} \label{3:1}
Let $K$ be a zero-dimensional compact space and let $L=K\cup N$ be its countable discrete extension.

\begin{enumerate}[(a)]
\item There is a continuous retraction $r$ from $L$ onto $K$ if and only if there
is a mapping  $N\ni n\mapsto x_n\in K$ such that $\lim_n(\delta_{x_n}-\delta_n)(C)=0$
for every clopen set $C\sub L$.
\item
There is an extension operator $E:C(K)\to C(L)$
 if and only if there is a bounded sequence of signed measures $(\mu _n)_{n \in N}$ on $K$ such that
 $\lim_n(\mu_{n}-\delta_n)(C)=0$
for every clopen set $C\sub L$.
 \end{enumerate}  
\end{lemma} 

For use in the next section we state here the following.

\begin{lemma} \label{3:2}
Let $\cA$ be an almost disjoint family of subsets of a countable set $N$ and suppose that $S\sub N$ is such that
$A\cap S$ is infinite for every $A\in\cA$. Consider the almost disjoint family $\cB=\{A\cap S: A\in\cA\}$
of subsets of $S$ and the bijection $\alpha:\cB\to\cA$ given by $\alpha(B)\cap S=B$. 

\begin{enumerate}[(a)]
\item If we declare that the sets $\{B\}\cup \big(\alpha(B)\sm F\big)$ with $F$ finite form a local base at $B\in K_\cB$
then $K_\cB\cup (N\sm S)$ is a countable discrete extension of $K_\cB$.
\item The space $K_\cB\cup (N\sm S)$ is homeomorphic to $K_\cA$.
\item If there is a bounded sequence $(\nu_n)_{n\in N\sm S}$ in $M(K_\cA)$ of measures supported by $S\cup\{\infty\}$ 
such that  $\lim_{n\in N\sm S} (\nu_{n}-\delta_n)(C)=0$ for every clopen set $C\sub K_\cA$ then
$\JL_p(\cB)$ is isomorphic to a complemented subspace of $\JL_p(\cA)$ for any $p\in (1,\infty]$.
\end{enumerate}
\end{lemma}

\begin{proof}
For the first statement just note that if a set  $D\sub N\sm S$ is infinite then either $D\cap A$ is finite for every
$A\in\cA$ and $D$ clusters at $\infty$ or $D\cap A$ is infinite for some $A\in\cA$ and the elements of $D\cap A$
converge to $A\cap S$. Part $(b)$ follows from the fact that, clearly, $\alpha$ extends to a homeomorphism
$K_\cB \cup (N\sm S)\to K_\cA$. 

For the final assertion we may think, by $(b)$, that $K_\cB$ is a subspace of $K_\cA$:
we use Lemma \ref{3:1}(b) and the fact that $E:C(K_\cB)\to C(K_\cA)$ defined
by $Ef(n)=\int_{S\cup\{\infty\}} f\; {\rm d}\nu_n$ for $n\in N\sm S$ is an extension operator that is bounded in supremum norms
and does not change $\norm {f}'_p$ as all the measures $\nu_n$ are supported by $S$.
\end{proof}

\section{The same $\JL_p$-spaces from finitely branching trees.}\label{4}

We first note  the following

\begin{lemma} \label{4:1}
If $T\sub \omega^{<\omega}$  is a tree and $S$ is its subtree
then $K_S$ is a retract of $K_T$. 
\end{lemma}

\begin{proof}
The retraction is obtained by sending all the nodes from $T\sm S$ and all the branches
from $[T]\sm [S]$ to infinity.
\end{proof}

We fix the following notation for the rest of  this section. For any natural number $m\ge 2$ let

\begin{enumerate}[(i)]
\item $T_m$ denote the full $m$-adic tree $m^{<\omega}$; we write  $T_2=T$ for simplicity;
\item $\cA_m$ be the almost disjoint family of branches of $T_m$;
\item $K_m$ be the compact space defined by $\cA_m$.
\end{enumerate}

We shall consider some trees that may be obtained from the
dyadic tree $T$. For a natural number $k$ we write 
\[ T^{(k)}=\bigcup_{n\in \omega}\lev_{kn}(T).\]
Note that $T^{(k)} $ is another tree (though it is not a subtree of $T$);
 in particular $T^{(2)}$ consists of elements from all even levels of $T$.

 Given $x\in\omega^\omega$, $a_x$ will denote the basic clopen set of $K_{\omega^{<\omega}}$ that consists of $x$ and of all initial segments of $x$.

\begin{lemma}\label{4:2}
Given $k\ge 2$, let $\cB$ be the almost disjoint family related to the  branches of $T^{(k)}$.
Then $\JL_p(\cB)$ is isomorphic to a complemented subspace of $\JL_p(\cA_2)$
for any $p\in (1,\infty]$.
\end{lemma}

\begin{proof}
Following Lemma \ref{3:2}, to every $\sigma$ from some  level of $T$ 
 which is not divisible by  $k$,
we assign a measure $\nu_\sigma$ on the set of nodes of $T^{(k)}$.

Let $\sigma\in\lev_{nk+j}(T)$ where $0<j<k$. Writing $2^{k-j}$ for the set of all 0-1 sequences of length $k-j$, we set
\[ \nu_\sigma=\sum_{\eps\in 2^{k-j}} \delta_{\sigma\sfrown\eps} - \big(2^{k-j}-1\big)\cdot\delta_\infty.\]

We claim that for every clopen set $C$, $\big(\nu_\sigma-\delta_\sigma\big)(C)=0$
except for finitely many $\sigma$'s. Notice first that $\nu_\sigma(K_2) = 1 = \delta_\sigma(K_2)$. 
If $C$ is a clopen subset of $K_2$, then either $C$ or $K_2\setminus C$ is of the form 
\[ D = \bigcup_{i=1}^n a_{x_i}\bigtriangleup F_i\]
 with $F_i \sub T$ finite. If we take $\sigma\in T$ of larger height than any node in 
 $\bigcup_i F_i$ and any bifurcation node of the branches $x_i$, then
\[\nu_\sigma(D) = \delta_\sigma(D) = \begin{cases} 1 & \text{ if } \sigma\prec x^i \text{ for some }i,\\ 0 & \text{ otherwise.}\end{cases}\]

Hence, we can apply Lemma \ref{3:2}(c) and the proof is complete.
\end{proof}

In the result below we use the well-known Pe{\l}czy\'nski decomposition method. In one of its variants, it states that two Banach spaces $X$ and $Y$ are
isomorphic provided that the following three conditions hold:

\begin{enumerate}[PDM(1)]
\item $X\simeq X\oplus X$ and $Y\simeq Y\oplus Y$;
\item $X$ is isomorphic to a complemented subspace of $Y$; 
\item $Y$ is isomorphic to a complemented subspace of $X$.
\end{enumerate}
The proof is a simple computation. Indeed,  if $X\simeq Y\oplus B$ and $Y\simeq A\oplus X$, then
\[ X\simeq Y\oplus B \simeq Y\oplus Y \oplus B \simeq Y \oplus X \simeq A\oplus X \oplus X \simeq A\oplus X \simeq Y.\]

In what follows we will denote by $\JL^0_p(\cA)$ the subspace of $\JL_p(\cA)$ made of all functions that vanish at $\infty$. Notice for $p\in(1,\infty)$ this is $\JL_p(\cA)$ itself, while for $p=\infty$, it is a hyperplane of $\JL_\infty(\cA)$ isomorphic to $\JL_{\infty}(\cA)$.

\begin{theorem}\label{4:4}
For every $p\in (1,\infty]$ and every $m\ge 2$ the Banach space $\JL_p(\cA_m)$ is isomorphic to 
$\JL_p(\cA_2)$.
\end{theorem}

\begin{proof}
As $T$ is a subtree of $T_m$, the space $\JL_p(\cA_2)$ is isomorphic to a complemented subspace of $\JL_p(\cA_m)$ by Lemma \ref{4:1}.

For $m\ge 2$ take $k$ such that $2^k\ge m$ and note that $T_m$ is isomorphic to a subtree of $T^{(k)}$. Hence, 
$\JL_p(\cA_m)$ is isomorphic to a complemented subspace of $\JL_p(\cA_2)$
by Lemma \ref{4:2} and Lemma \ref{4:1}.

Let us  check that $\JL_p(\cA_m)\simeq \JL_p(\cA_m)\oplus \JL_p(\cA_m)$. 
We will define mutually inverse operators
\[U:\JL^0_p(\cA_m)\oplus \JL^0_p(\cA_m) \To \JL^0_p(\cA_m),\]
\[U^{-1}:\JL^0_p(\cA_m)\To \JL^0_p(\cA_m)\oplus \JL^0_p(\cA_m).\]
In order to define $U$ and $U^{-1}$ we need some preliminary notation. For every $n<\omega$, consider $z_n\in T_m$,  the node that has length $n$ and all coordinates are 0, $z_n = (0,0,\ldots,0)$. Let $S$ be the set of all nodes of the tree $T_m$ whose all coordinates are zero except for the last one. We enumerate this set as 
\[S = \{z_n{}^\frown k \in T_m : n<\omega, 0<k<m\} = \{s_p : p<\omega\}.\] 
In the following definition, $t$ and $r$ could be either finite or infinite sequences of $0,1,\ldots,m-1$. The finite sequences are viewed as isolated points of $K_{\mathcal{A}_m}$, and the infinite sequences are identified with branches of $T_m$, hence also points of $K_{\mathcal{A}_m}$.
\[U(f,g)(t) = \begin{cases}   f(r) & \text{ if } t=s_0{}^\frown r,\\ g(s_{p-1}{}^\frown r) & \text{ if } t = s_p^\frown r \text{ for some }p>0,\\ g(t) & \text{ if } t=z_n.
 \end{cases}\]
and $U^{-1}(h) = (f,g)$ where $f(r) = h(s_0^\frown r)$,  $g(s_q{}^\frown r) = h(s_{q+1}{}^\frown r)$, $g(z_n) = h(z_n)$.
 It is easy to check that these operators are bounded for every $p$.
 
We conclude the argument using  the decomposition method.
 \end{proof}

\section{Measuring the  distance}\label{5}
As in the previous section, we denote by
$\cA_m$  the almost disjoint family of branches of $m^{<\omega}$
and   $K_m$ is the compact space defined by $\cA_m$.
The proof consists of three independent steps. First we discretize the images of the basis vectors. Next we show that one coefficient must be large. Finally, we exploit the tree structure to obtain the lower estimate.

\begin{theorem}\label{estimateAnembedding}
Let $n>m\ge 2$ and $1<p<\infty$.
	Suppose that $U:\JL_p({\mathcal{A}_n})\To  \JL_p({\mathcal{A}_m})$ is an isomorphic
	embedding.  Then 
	\[\|U\|\|U^{-1}\|\ge \frac{ n^{1-1/p}}{4(m+1)}.\]
\end{theorem}

\begin{proof}
For $x\in n^\omega$ we have a clopen set $C(x)=B(x)\cup\{x\}$ in $K_n$ and we  write
$e_x$ for its characteristic function. Then $e_x$ is a norm-one element of $\JL_p(\cA_n)$.
For clarity, we denote by $f_y$ the element of $\JL_p(\cA_m)$ corresponding to $y\in m^\omega$.
 We can assume that $\|U^{-1}\|=1$ and concentrate on estimating $\|U\|$.
We fix (small) $\varepsilon>0$. 
\medskip

\noindent{\bf Part 1.} Note that every element of $\JL_p(\cA_m)$ can be
approximated in norm by a step functions. Hence to every $x\in. n^\omega$ we can associate

\begin{enumerate}[(i)]
\item a finite set $F_x\sub m^\omega$,
\item a function $\vf_x:F_x\to{\mathbb Q}$,
\item $l_x\in\omega$ such that for distinct $y,y'\in F_x$ we have $y|l_x\neq y|l_x$.
\item a function $\xi_x\in c_{00}(m^{<\omega})$ with rational values,
\end{enumerate} 
so that 
\[ \|Ue_x-h_x\|_p<\eps, \mbox{ where } h_x=\sum_{y\in F_x} \vf_x(y)f_y +\xi_x.\]

By the Baire category argument, we can freeze parameters from countable sets to find
$\tau\in n^{<\omega}$ such that the set $X\sub n^\omega$ of $x$'s with these parameters fixed
is uncountable and dense in $[\tau]$. In other words, for $x\in X$, we have
$|F_x|=k$, $l_x=l$, $\xi_x=\xi\in c_{00}(m^{<\omega})$ and, whenever we list
elements of $F_x$ as $y_0,\ldots, y_{k-1}$ in the lexicographic order then
$\vf_x(y_i)=\vf(i)$ for $i<k$ for some fixed $\vf$.
\medskip
    
\noindent    \textbf{Part 2.} 
There is $i<k$ such that $|\vf(i)|>1/2-\eps$.
\medskip

To prove the claim we use the fact that $X$ is uncountable. By the $\Delta$-system lemma
we find an uncountable $Y\sub X$ such that the family $\{F_x:x\in Y\}$ forms a $\Delta$
system. Then we can find two incomparable nodes $\tau_1,\tau_2$ extending $\tau$ such that
$Y\cap '[\tau_1]$ and $Y\cap [\tau_2]$ are still uncountable.

Consider any $d\in\omega$, pick distinct $x_1^1,\ldots, x_d^1\in Y\cap [\tau_1]$
and $x_1^2,\ldots, x_d^2\in Y \cap [\tau_2]$ and consider the functions
\[ g_1=\frac{1}{d}\sum_{j=1}^d e_{x_j^1},\quad 
g_2=\frac{1}{d}\sum_{j=1}^d e_{x_j^2},\]
\[ h_1=\frac{1}{d}\sum_{j=1}^d h_{x_j^1},\quad 
h_2=\frac{1}{d}\sum_{j=1}^d h_{x_j^2},\]

The supremum norm of $g_1-g_2$ equals 1, so $\|g_1-g_2\|_p\ge 1$. Hence
$\|Ug_1-Ug_2\|_p\ge 1$ (as $\|U^{-1}\|=1$). The point is that 
$\|Ug_1-Ug_2\|_p'< 1$ if $d$ is large enough --- see Lemma \ref{additional} below. Hence  
\[ 1\le  \|Ug_1-Ug_2\|_\infty\le \|h_1-h_2\|_\infty+2\eps \le \max_i 2|\vf(i)|+2\eps,\] 
and the claim follows.
\medskip

\noindent {\bf Part 3.}
From now on, we can assume that large constant $\vf(i)$ we found above is positive. 
We denote it by  $c=\vf(i)$.
It is now enough to estimate $\|U\|$ assuming that there is $\sigma_0\in m^{<\omega}$ such that
for every $x\in X$  there is a unique element  $y_x\in [\sigma_0]$ such that 
\[ \|\big(Ue_x\big)\cdot \chi_{[\sigma_0]}-cf_{y_x}\|_\infty<\eps.\]

Pick $x_j\in [\tau\sfrown j]\cap X$ for every $j<n$ and write $f_{x_j}=f_j$ for simplicity. 
For every $\sigma$ with $\sigma_0\preceq \sigma$ write 
\[ S(\sigma)=\{j<n: \sigma\prec y_j\},\quad 	N(\sigma)=\{j<n: \sigma\not\prec y_j\}.\]
We have $N(\sigma_0)=\emptyset$ so there is a maximal $\sigma\succeq\sigma_0$ having the property
\[ \big| N(\sigma)\big|< \frac{n}{m+1}.\]

Write $\rho$ for the characteristic function of the set $R=\{t\in n^{<\omega}: t\preceq \tau\}$.
We examine the values which $U\rho$ assumes at 
 the  nodes $\sigma\sfrown i$.
 
 Suppose that there is $i<m$ such that 
 $|U\rho(\sigma\sfrown i)|\ge c/2$; let $U\rho(\sigma\sfrown i)\ge c/2$ (the negative case is analogous).
  Then for every $k\in N(\sigma\sfrown i)$  we have
\[ |Ue_{x_k}(\sigma\sfrown i)|<\eps \mbox{ so } U(e_{x_k}-\rho)(\sigma\sfrown i)\ge c/2-\eps.\]
 As $\big| N(\sigma\sfrown i)\big|\ge n/(m+1)$
by the maximality of $\sigma$, it follows that
\[ \mbox{ for } h=\sum_{j\in N(\sigma\sfrown i)} (e_{x_j}-\rho) \mbox{ we have} Uh(\sigma\sfrown i) \ge (c/2-\eps)\frac{n}{m+1}.\] 

As $\|h\|_\infty=1$ and $\|h\|_p'\le n^{1/p}$, we have $\|h\|_p\le n^{1/p}$ so
\begin{equation}\label {est}
\|U\|\ge (c/2-\eps)\frac{n^{1-1/p}}{m+1}.
\end{equation}

In the opposite case, note that 
\[ \big| S(\sigma)\big| > n-\frac{n}{m+1}=n\frac{m}{m+1},\]
and we choose  $i<m$ such that 
\[ \big| S(\sigma\sfrown i)\big|\ge \frac{n}{m+1}. \]
By a similar argument,
\[ \mbox{ for } h=\sum_{j\in S(\sigma\sfrown i)} (e_{x_j}-\rho) \mbox{ we have}. |Uh(\sigma\sfrown i)|\ge (c/2-\eps)\frac{n}{m+1},\] 
and we again get the estimate as in (\ref{est}).

Finally, the assertion follows from  the estimate (\ref{est}), since $\eps$ can be arbitrarily small (recall that $c\ge 1/2 -\eps$).
\end{proof}	

To complement the above proof we record  the following elementary observation.

\begin{lemma}\label{additional}
Let $v_1,\ldots, v_{2d}\in\ell_p$ be  vectors for which  there are $a_1,\ldots, a_{2d}\in\ell_p$
such that
\begin{enumerate}[(i)]
\item $\|a_i\|_p\le r$ and $\|v_i-a_i\|_p<\eps$ for every $i$;
\item the supports of $a_i$ are finite sets forming a $\Delta$-system and all $a_i$ have the same values on
their common root.
\end{enumerate}
Then the vectors  $w_1=(1/d)\sum_{i=1}^dv_i, w_2=(1/d)\sum_{i=d+1}^{2d}v_i$ satisfy
\[\|w_1-w_2\|_p\le 2d^{1/p-1}\cdot r+2\eps.\]
\end{lemma}	

\begin{proof}
If we write $b_1=\sum_{i=1}^d a_i$, $b_2=\sum_{i=d+1}^{2d} a_i$ then 
$b_1-b_2$ is the sum of $d$ disjoint blocks of norm bounded by $2r$ so 
$\|b_1-b_2\|_p\le 2d^{1/p}\cdot r$.
Moreover, 
$\|w_1-w_2\|_p\le \|1/d(b_1-b_2)\|_p+2\eps$, and the assertion follows. 
\end{proof}

The above theorem has the following $\infty$-version.
Here we write $C_0(K_m)=\{g\in C(K_m): g(\infty)=0\}$.

\begin{theorem}\label{estimateAnembedding2}
Let $n>m\ge 2$ and 
	suppose that $U: C(K_n)\To  C_0(K_m)$ is an isomorphism
	embedding.  Then 
	\[\|U\|\|U^{-1}\|\ge \frac{ n}{4(m+1)}.\]
\end{theorem}	
	
\begin{proof}
Let us briefly examine the previous proof in the present case. 
Part 1 remains the same (thanks to the fact that $U$ takes values in $C_0(K_m)$).
The claim of Part 2 is now immediate since $U(e_x- e_{x'})$ must have the supremum norm 
$\ge 1$ whenever $x\neq x'$. Finally, the functions $h$ we used in Part 3 have norms 1 so
we can replace (\ref{est}) by $\|U\|\ge (c/2-\eps)\frac{n}{m+1}.$
\end{proof}

\begin{corollary}   
	If $\Upsilon$ is the tree for which every node at level $n$ has $n+1$ immediate successors, 
	then $\JL_p({\mathcal{A}_\Upsilon})$ is not isomorphic to a subspace of $\JL_p({\mathcal{A}_2})$ for every $p\in (1,\infty]$.
\end{corollary}

\begin{proof}
	Consider $p<\infty$. The functions that are supported on an everywhere exactly $n$-branching subtree of $\Upsilon$ provide an isometric copy of  $\JL_p({\mathcal{A}_n})$ inside $\JL_p({\mathcal{A}_\Upsilon})$. So, if  $\JL_p({\mathcal{A}_\Upsilon})$ was isomorphic to a subspace of $\JL_p({\mathcal{A}_2})$, then there would be a universal constant for isomorphisms of all $\JL_p({\mathcal{A}_n})$ inside $\JL_p({\mathcal{A}_2})$, which contradicts Theorem~\ref{estimateAnembedding}, because
	\[\lim_n  \frac{ n^{1-1/p}}{4(m+1)} = +\infty.\]

The argument for $p=\infty$ is similar: we refer to Theorem \ref{estimateAnembedding2}
and the fact  that the hyperplane $C_0(K_2)$ is isomorphic to $C(K_2)$.
\end{proof}


Theorem~\ref{Bairetreeembedding} below is due to Marciszewski and Pol \cite{MP09}, 
but we sketch an alternative  argument similar to that of Theorem~\ref{estimateAnembedding}.

\begin{theorem}\label{Bairetreeembedding}
	If $\Upsilon$ is a finitely branching tree, then $\JL_p(\cA_{\omega^{<\omega}})$ is not isomorphic to a subspace of $\JL_p(\cA_{\Upsilon})$.
\end{theorem}

\begin{proof}
Let us briefly discuss the case $p<\infty$.
We can repeat the first two parts of the argument used in \ref{estimateAnembedding}
but we need a different argument for Part 3.

We suppose that $U: \JL_p(\cA_{\omega^{<\omega}}) \To\ JL_p(\cA_{\Upsilon})$
is a norm-increasing  isomorphic embedding and arrive at  a contradiction by showing that
$\|U\|=+\infty$.

It is now enough to estimate $\|U\|$ assuming that there is $\sigma_0\in \Upsilon$ such that
for every $x\in X$  there is a unique element  $y_x\in [\sigma_0]$ such that 
$\|\big(Ue_x\big)\cdot \chi_{[\sigma_0]}-cf_{y_x}\|_\infty<\eps$.
Recall that we can assume that  $c-\eps>0$.

We now pick $x_n\in X\cap [\tau\sfrown n]$ for every $n$. 
Write $\rho$ for the characteristic function of the set $R=\{t\in \omega^{<\omega}: t\preceq \tau\}$.
Note that the sequence of $e_{x_n}$ converges weakly in $\JL_p(\cA_{\omega^{<\omega}})$ to $\rho$.
Moreover, we can assume that the corresponding branches $y_n$ converge to $y\in [\Upsilon]$
in the natural topology of $[\Upsilon]$. Since $Ue_{x_n}$ converge weakly in $\JL_p(\Upsilon)$ we
get the following.
\medskip

\noindent {\sc Claim.} $U\rho(\sigma)\ge c-\eps$ for every $\sigma\in \Upsilon$ satisfying
$\sigma_0\preceq \sigma\prec y$.
\medskip

Consider $\sigma$ with $\sigma_0\prec\sigma\prec y$ and suppose that
\[ 	N(\sigma)=\{n: \sigma\not\prec y_{x_n}\},\]
has  $m$ elements. Then $h=\sum{j\in N(\sigma)}(e_{x_n}-\rho)$
is an element of $\JL_p(\omega*{<\omega})$ with $\|h\|_p\le m^{1/p}$ and
$Uh(\sigma)\ge m(c-\eps)$. Hence,
$\|U\|\ge m^{1-1/p}(c-\eps)$. Since we can take $\sigma$ such that $m$ is arbitrary large,
it follows that $\|U\|=+\infty$.
\end{proof}

\section{Small $\JL$-spaces}\label{6}

The main objective of this section is to show that a weakening of Martin's axiom allows us to show any two small AD families have isomorphic Johnson-Lindenstrauss spaces. 
To do so we introduce an operation between two almost disjoint families of the same cardinality, $\mathcal{A}$ and $\cA'$ of subsets of
 two disjoint countable sets $N$ and $N'$, respectively. Suppose that $\theta:\cA\to\cA'$ is a fixed bijection.
We  then write $A_\theta =A\cup\theta(A)$ and consider
$$ \cA\oplus_\theta\cA'=\{A\cup\theta(A):A\in\cA\},$$
which is an almost disjoint family on $N\cup N'$. 

Note that the Alexandrov-Urysohn
space $K_{\cA\oplus_\theta\cA'}$ may be seen as the image of the related quotient map
on $K_\cA\cup K_{\cA'}$ gluing every $A\in\cA$ with $\theta(A)\in\cA'$, and identifying
the corresponding points at $\infty$. Hence, we can think that $K_\cA$ and $K_{\cA'}$ are subspaces
of $K_{\cA\oplus_\theta\cA'}$. The operator that corresponds to the restriction to $K_{\cA'}$ is
$$ R':  \JL_p(\cA\oplus_\theta \cA')\to \JL_p(\cA'),$$
 given by $R'f(\theta(A))=  f(A\cup\theta(A))$ for $A\in\cA$, $R'f(n)=f(n)$ for $n\in N'$, and $R'f(\infty)=f(\infty)$). 
A bounded operator
\[ E':\JL_p(\cA')\to \JL_p(\cA\oplus_\theta\cA'),\]
is an extension operator if $R'\circ E'$ is the identity. In a similar way, we have the restriction $R:  \JL_p(\cA\oplus_\theta \cA')\to \JL_p(\cA)$ and extension operators $E:\JL_p(\cA)\to \JL_p(\cA\oplus_\theta\cA')$.

\begin{lemma}\label{5:1}
Consider an almost disjoint family $\cA\oplus_\theta\cA'$ as above and any $p\in (1,\infty]$. Suppose
that there are extension operators
\[  \JL_p(\cA')\to \JL_p(\cA\oplus_\theta\cA') \mbox{ and } \JL_p(\cA)\to \JL_p(\cA\oplus_\theta\cA').\]
 Then $\JL_p(\cA)\simeq \JL_p(\cA')$.
 \end{lemma}

\begin{proof}
If $E':\JL_p(\cA')\to \JL_p(\cA\oplus_\theta\cA')$ is an extension operator then
$P'=E'\circ R' $ is a projection since $R' \circ E'$ is the identity.
Moreover, $c_0(N)$ is the kernel of $P'$ so 
\[\JL_p(\cA\oplus_\theta \cA')\simeq c_0(N) \oplus \JL_p(\cA')\simeq  \JL_p(\cA').\]
By symmetry, $\JL_p(\cA\oplus_\theta \cA')\simeq  \JL_p(\cA)$, and we are done.
\end{proof}

In our present setting, Lemma \ref{3:1}
can be formulated as follows.

\begin{lemma}\label{5:2}
Consider an almost disjoint family $\cA\oplus_\theta\cA'$ as above and suppose
that there is a mapping $N\ni n\mapsto \nu_n$ such that

\begin{enumerate}[(i)]
\item every $\nu_n$ is a measure concentrated on $N'$ and $\nu_n(N')=1$;
\item $\sup_n\|\nu_n\|<\infty$;
\item $\nu_n(F)=0$ for every finite $F\sub N'$ and almost all $n\in N$;
\item $ \big(\delta_n-\nu_n\big)(A\cup\theta(A))=0$ for every $A\in\cA$ and almost all $n\in N$.
\end{enumerate}
Then there is an extension operator 
$$ E:\JL_p(\cA')\to \JL_p(\cA\oplus_\theta\cA'),$$
such that  $\norm{E}\leq \sup_n\|\nu_n\|$.
\end{lemma}

\begin{proof}
We  define $E$
by setting,  for every $g\in \JL_p(\cA')$,   $Eg(A\cup\theta(A))=g(\theta(A))$ for $A\in \cA$. $Eg(n') = g(n')$ for $n'\in N'$, and 
\[ Eg(n)= \int_{N'} g\;{\rm d}\nu_n,  \mbox{ for } n\in N.\]

The only subtle point is that $Eg$ is  continuous on $K_{\cA\oplus_\theta\cA'}$. 
As every $g\in \JL_p(\cA')$ can be uniformly approximated by simple continuous functions,
it is enough to consider $g=\chi_C$ for some clopen set $C$.
Any point of the form $A\cup\theta(A)$ in this space is a common limit of the sequences $A\subset N$ and $\theta(A)\subset N'$. Property (iv) implies that each clopen set $C_A = A\cup \theta(A)\cup\{A\cup\theta(A)\}$ satisfies that 
\[\nu_n(C_A) = \delta_n(C_A) = \begin{cases} 1 & \text{ if } n\in A \\ 0 & \text{ otherwise,} \end{cases}\]
for all but finitely many $n\in N$. Moreover, by property (iii), for any given finite set $F\subset N\cup N'$, $\nu_n(F) = \delta_n(F)=0$ for all but finitely many $n\in N$. By property (i), $\nu_n(K_{\cA\oplus_\theta\cA'})=1 = \delta_n(K_{\cA\oplus_\theta\cA'})$. Every clopen set $C$ (or its complement) is a finite modification of a finite union of sets $C_A$, so  $\delta_n(C)=\nu_n(C)$ for almost all $n$ whenever $C$ is clopen. This implies the continuity of $E\chi_C$ because whenever $\lim_k n_k = \xi$ is a nontrivial limit we will have
\[
\lim_k E\chi_C(n_k) = \lim_k \nu_{n_k}(C) = \lim_k \delta_{n_k}(C) = \delta_{\xi}(C) = \chi_C(\xi) = E \chi_C (\xi).\]
\end{proof}

We are ready for the main result of this section, which requires a weak version of Martin's Axiom. 
For the convenience of the reader we recall the relevant terminology; see Peng \cite{Peng}
for a short survey on versions of Martin's axiom and further references. Let $(\mathbb P,\leq)$ be a partially ordered set, and fix a set $D\subseteq \mathbb P$. 
The set $D$ is called  \textit{open} if for any $p\in D$ and $q\leq p$ we have $q\in D$. The set $D$ is called \textit{dense} if for any $p\in \mathbb P$ there is $q\in D$ such that $q\leq p$. 
A set $L\sub {\mathbb P}$ is called linked if for any $p_1,p_2 \in L$ there is $q\in \mathbb P$ such that $q\leq p_1, p_2$. A filter is a subset $G\subseteq \mathbb{P}$ such that if $p\in G$ and $p\leq q$ then $q\in G$ and if $p,q\in G$ there is $r\leq p,q$ such that $r\in G$.
 Finally, $\mathbb P$ is called $\sigma$-linked if 
 it is a countable union of linked parts. 
 For a cardinal $\kappa<\mathfrak c$,  the $\sigma$-linked version of Martin's axiom for $\kappa$ reads as follows.

\medskip

MA$_\kappa(\sigma \text{-linked})$: {\em For any $\sigma$-linked poset $\mathbb P$, 
and any collection $\{D_\alpha:\alpha<\kappa\}$ of dense open subsets of $\mathbb P$ there is a 
filter $G\subseteq \mathbb P$ such that  $D_\alpha\cap G\neq \emptyset$
for every  $\alpha<\kappa$.}

\medskip   


The following is the crucial application of the axiom; recall that we still consider
two almost disjoint families $\cA$ and $\cA'$ on two disjoint countable sets $N$ and $N'$, respectively.

\begin{lemma}[MA$_\kappa(\sigma \text{-linked})$]\label{5:3}
If $|\cA|=|\cA'|=\kappa$ and $\theta:\cA\to\cA'$ is any bijection then for every  $p\in (1,\infty]$ 
 there is an extension operator $E: \JL_p(\mathcal A')\to \JL_p({\mathcal A\oplus_\theta \mathcal A'})$, such that $\norm{E}\leq 3$.
\end{lemma}

\begin{proof}
To prove this,  we can assume that $N=\omega$ and
fix a bijection $\varphi:\omega \to N'$.
For every $n\in\omega$ write $I_n=\{k:k<n\}$ and $J_n=\{\vf(k): k<n\}$.
Define $M$ as the set of measures $\mu$ on $N'$ such that either $\mu=\delta_i$ for some
$i\in N'$ or $\mu=\delta_i-\delta_j+\delta_k$ for different $i,j,k\in N'$.
Thus for every $\mu\in M$ we have $\mu(N')=1$ and $\|\mu\|\le 3$.

Consider the partial order $\mathbb P$ consisting of all conditions $p=(\nu_p,n_p,\cF_p)$ such that

    \begin{enumerate}[(i)]
        \item $n_p\in \omega$;
        \item $\cF_p\in [\mathcal A]^{<\omega}$;
        \item $\nu_p$ is a  function $I_{n_p}\to M$;
        \item $A\cap B\sub I_{n_p}$ and $\theta(A)\cap \theta(B)\sub J_{n_p}$ whenever $A,B\in\cF_p$ are distinct.
\end{enumerate}
 
We declare that $q\leq p$ if $n_p\leq n_q$, $\cF_p\subseteq \cF_q$, $\nu_q$ is an extension of $\nu_p$, and 
\[ (*)\quad \nu_q(j)(\theta(A))=\delta_j(A)  \mbox{ and }  \nu_q(j)(\{m\})=0,\]
 whenever $n_p\le j<n_q$, $ A\in \cF_p$, and $m\in J_{n_p}$. 

We need to check that $\mathbb P$ is $\sigma$-linked.
For $n\in\omega$, and $\nu: I_n\to M$. Consider \[ L_{n,\nu}=\{p\in {\mathbb P}: n_p=n, \nu_p=\nu\}.\]

It is easy to see that $\mathbb P=\bigcup_{n,\nu}L_{n,\nu}$. There are only countably many different $n$ and $\nu$, so it is enough to show every $L_{n,\nu}$ is linked; that is, every $p,q\in L$ can be extended to some $r\in {\mathbb P}$.

\newcommand{\wtn}{\widetilde{\nu}}
To this end fix $n\in\omega$, $\nu: I_n\to M$ and $p,q\in L_{n,\nu}$.
 Find $k\ge n$ minimal with respect to the property that for every distinct $A,B\in \cF_p\cup \cF_q$, $A\cap B\sub I_k$ and $\theta(A)\cap \theta(B)\sub J_k$. Before continuing we can make the following observations:
 
For every $A\in\cF_p\cup \cF_q$ there exists $i_A\in \theta(A)\sm (\bigcup_{B\in (\cF_p\cup \cF_q)\setminus \{A\}} \theta(B)\cup J_k)$.
There is $h\in N'\setminus (\bigcup_{C\in\mathcal{F}_p\cup \mathcal{F}_q}\theta(C)\cup J_k)$. Finally, notice that for any $j\in I_k\sm I_n$, $\abs{\{A\in \cF_p\cup\cF_q:j\in A\}}\leq 2$. This is so because if there are $A_0,A_1,A_2\in \cF_p\cup\cF_q$ different such that $j\in  A_0 \cap A_1 \cap A_2$, then there are different $i,i'\in \{0,1,2\}$ such that $A_i,A_{i'}\in \mathcal F_p$ or  $A_i,A_{i'}\in \mathcal F_q$. But in either case $A_i\cap A_{i'}\subseteq I_{n}$, so that $j\in I_n$, which is a contradiction. 

We will define $\wtn:I_{k}\to M$ that extends $\nu$ and fulfills (*) by cases. Consider $j\in I_k\sm I_n$.

\begin{itemize}
\item If  $j$ belongs to no $A\in\cF_p\cup\cF_q$
then we define $\wtn(j)=\delta_h$.
\item If $j\in A$ for exactly one $A\in\cF_p\cup\cF_q$
then we define $\wtn(j)=\delta_{i_A}$.
\item If $j\in A\cap B$ where $A\in \cF_p$ and $B\in \cF_q$
then we define $\wtn(j)=\delta_{i_A}+\delta_{i_B}-\delta_h$. 
\end{itemize}

Then (*) holds so   $r=(k, \cF_p\cup\cF_q, \widetilde{\nu})$ satisfies $r\le p,q$. 

Now consider for every $n\in\omega$ and for any $A\in \mathcal A$, 
\[  D_n=\{p\in {\mathbb P}: n_p> n\} ,\quad  D_A=\{p\in{\mathbb P}: A\in \cF_p\}..\]
There are $\kappa$ such sets.
\medskip

\noindent \textsc{Claim.} For every $A\in \mathcal A$, $D_A$ is dense in $\mathbb P$.
\medskip

Fix $p\in \mathbb P$ and assume $A\notin \mathcal F_p$. Find $k\ge n_p$ minimal with respect to the property that for every $B\in \mathcal F_p$, $A\cap B\subseteq I_{k}$ and $\theta(A)\cap \theta(B)\subseteq J_k$. Define $\wtn:I_{k}\to M$ by cases as we did previously. If $r=(k,\mathcal F_p\cup \{A\},\wtn)$, then $r\le p$ and $r\in D_A$. This establishes the claim

The proof that for every $n\in \omega$, $D_n$ is dense is analogous.

Notice that any filter $G\sub\mathbb P$
meeting every $D_n$ defines a map $\nu:\omega\to M$. Consider a filter $G\sub\mathbb P$ meeting every $D_n$ and every $D_A$.
To complete the proof, we only need to show that 
the sequence of measures $(\nu_n)_{n\in \omega}$ such that $\nu_n=\nu(n)$ satisfies the conditions of Lemma \ref{5:2}.

 The first two conditions have already been checked. 
For the third condition.
If $F\sub N'$ is finite then there is $n\in \omega$ such that $F\sub J_n$. Find $p\in D_{n}\cap G$. Now consider $j>n_p$ there must be $q\in G$ such that $j\leq n_q$. As $G$ is a filter there is $r\leq p,q$, which implies $\nu(j)(\{m\})=\nu_r(j)(\{m\})=0$ for any $m\in F\subseteq J_n$ by (*). In an analogous way the fourth condition of the lemma will follow from (*) as well.
\end{proof}

 Lemma \ref{5:1} and Lemma \ref{5:3} are enough to show.

\begin{theorem}\label{last}
Under MA$_\kappa(\sigma \text{-linked})$,
if $\mathcal{A},\mathcal B$  are two almost disjoint families of size $\kappa<\mathfrak{c}$ 
then the spaces $\JL_p(\cA)$ and $\JL_p(\cB)$ are isomorphic for $p\in (1,\infty]$.
\end{theorem}

\section{Some old results and some new open problems.}\label{sectionproblems}

The following are the most natural questions that remain open for us:

\begin{problem}
	Characterize when two trees $T, T'$ satisfy that $K_{\cA_T}$ and $K_{\cA_{T'}}$ are homeomorphic, when $K_{\cA_T}$ is a continuous image of $K_{\cA_{T'}}$, and when $\JL_p({\cA_T})$ is isomorphic to $\JL_p({\cA_{T'}})$.
\end{problem}

\begin{problem}
	Solve the previous problem for the particular case when $T=2^{<\omega}$, $T=\omega^{<\omega}$ or when $T$ is the tree in which every node of height $n$ has $n+1$ immediate successors.
\end{problem}

There is natural problem along this line of thought in which we could not make significant progress, that we would like to mention. The first named author and Todorcevic \cite{AT18} developed another method of constructing almost disjoint families from trees in a natural way
that give  rise to non-homeomorphic compacta. Do these spaces have also non-isomorphic spaces of continuous functions? In the rest of the section, we give a brief self-contained exposition of the relevant facts from \cite{AT18} needed for the precise formulation and context of our problem. We consider the following topological invariant.

\begin{definition}\label{2:8}
	The open degree $\odeg(K)$ of a compact space $K$ is defined by saying that
	$\odeg(K)\le n$
	if  there is a countable family $\cV$ of open subsets of $K$ such that
	whenever $x_0,\ldots, x_n$ are pairwise distinct points in $K$, there are $V_0,\ldots, V_n\in\cV$
	with $x_i\in V_i$ for every $i$ and $\bigcap_iV_i=\emptyset$. 
\end{definition}

A different characterization is that $\odeg(K)\leq n$ if and only there exist two continuous surjections $f:L\To K$ and $g:L\To M$ between compact space such that $M$ is metrizable and $|g^{-1}(x)|\leq n$ for all $x\in M$ \cite[Theorem 7.2]{AvilesKrupskiLSigma}. From this, it is clear that $\odeg(K_1)\leq\odeg(K_2)$ when there is a continuous surjection from $K_2$ onto $K_1$.

Note that $\odeg(K_T)\le 2$ for every subtree $T$ of $\omega^{<\omega}$. This is indicated by the family consisting of all singletons of isolated points of $K_T$, its complements, and all the sets
of the form $[\tau]\cup \{s\in \omega^{<\omega}:\tau\preceq s\}$ for $\tau\in T$. 
If $p_0,p_1,p_2\in K_T$ are distinct then for instance, $p_0,p_1$ are not $\infty$ and this pair can be separated.

It turns out that splitting the branches into pieces in a suitable way, one can obtain AD families that define compacta of higher degree.

Consider the $m$-adic tree $m^{<\omega}$ and for every $x\in m^\omega$ and every $i\in\{0,\ldots,m-1\}$ write
\[ A_x^i=\{\sigma\in m^{<\omega}: \sigma\sfrown i\prec  x\}.\]
These sets $A_x^0,A_x^1,\ldots,A_x^{m-1}$ form a partition of the branch below $x$. 

The following is a particular case of \cite[Theorem 33]{AT18} applied to $K_\infty(\mathfrak{Q})$ for the partition $\mathfrak{Q}=\{\{0\},\{1\},\ldots,\{m-1\}\}$. For the convenience of the reader, we include a proof that avoids the particular language of \cite{AT18}.

\begin{theorem}\label{2:9}
	Let $\cA^{[m]}$ be the almost disjoint family of all sets of the  form $A_x^i$ which are infinite, for $x\in m^\omega$, $i<m$.
	Then $\odeg (K_{\cA^{[m]}})=m+1$. In particular, $K_{\cA^{[m]}}$ is not a continuous image of $K_{\cA^{[n]}}$ if $n<m$, and none of these compact spaces is a continuous image of $K_{T}$ for any subtree $T$ of $\omega^{<\omega}$.
\end{theorem}

\begin{proof}
	The proof of $\odeg(K_{\cA^{[m]}})\le m+1$ is similar to the proof of $\odeg(K_T)\leq 2$ above. Denote $\cA^{[m]}=\mathcal A$. The family of open sets consists of all isolated points and their complements, plus all sets of the form $W_\tau = \{A^i_x : x\in [\tau]\}\cup\{s\in\omega^{<\omega} : \tau\preceq s\}$. 
	Take distinct points $p_0,\ldots,p_{m+1}\in K_{\cA}$. If one  point $p_k$ is isolated, take $V_k=\{p_k\}$ and the rest of open sets as the complement. Otherwise all points are either $\infty$ or of the form $A^i_x$. There must be two points of the form $A^i_x$ and $A^j_y$ with $x\neq y$. Take $\tau,\tau'$ such that $\tau\prec x$ and $\tau'\prec y$ and $[\tau]\cap [\tau']=\emptyset$, take $W_\tau$ and $W_{\tau'}$ as neighborhoods of those points and the rest as you wish.
	
	Let us check that $\odeg(K_\cA)>m$. Suppose that there is a countable family $\cV$ of open subsets of
	$K_\cA$ separating every ($m+1$)-tuple $p_0,p_1,\ldots p_m$ as in Definition \ref{2:8}.
Define \[G=\{x\in m^{\omega}:\forall i<m, \abs{n:x(n)=i}=\omega\}.\]

Notice that for any $i<m, k<\omega$, the set 
\[G_{i,k}=\{x\in m^{\omega}: \exists n>k \ {x(n)=i} \},\]
 is open and dense in $m^{\omega}$. Moreover, $G=\bigcap_{i<n,k<\omega}G_{i,k}$   so $G$ is a dense $G_\delta$-set. 
 Furthermore, for every $x\in G$ and $i<m$, $A^i_x$ is infinite.
    
	In particular, for every $x\in G$ there are $V_x^0, V_x^1,\ldots,V_x^{m-1}, V_x^\infty\in\cV$ such that $A_x^i\in V_x^i$ for $i<m$, $\infty \in V_x^\infty$ and $V_x^0\cap \cdots\cap V_x^{m-1}\cap V_x^\infty=\emptyset$.
	Then for every $x\in G$ there is a finite set $F_x$ such that $A_x^i\setminus F_x\sub V_x^i$ for $i<m$.
	Since $\cV$ is countable, by the Baire category theorem there is $X\sub G$ 
	and a finite set $F\sub m^{<\omega}$
	such that
	
	\begin{enumerate}[(i)]
		\item  $V_x^i=V^i\in\cV$ for $i<m$, $V_x^\infty =V^\infty\in\cV$, $F_x=F$ for every $x\in X$;
		\item $X$ is dense in $[\tau]$ for some $\tau\in m^{<\omega}$;
		\item the set $D_\tau=\{\sigma \in m^{<\omega}: \tau\prec \sigma\}$ is disjoint from $F$.
	\end{enumerate}
	
	If we take any $\sigma\in D_\tau$ there are $x^i\in X$ such that $\sigma\sfrown i\prec x^i$ for $i<m$. Then $\sigma\in A_x^i\setminus F \sub V^i_x=V^i$ for all $i<m$. 
	This means that $D_\tau\sub \bigcap_{i<m}V^i$. Therefore $D_\tau\cap V^\infty= \emptyset$ because $V^0\cap \cdots\cap V^{m-1}\cap V^\infty=\emptyset$ . This contradicts that $V^\infty$ is a neighborhood of
	$\infty$.
\end{proof}


\begin{problem}
	Is $C(K_{\cA^{[m]}})$ isomorphic to $C(K_{\cA^{[n]}})$ when $n\neq m$?
\end{problem}

\bibliographystyle{siam}
\bibliography{references}

@misc{KR25,
      title={Almost disjoint families and some automorphic and injective properties of $\ell_\infty/c_0$}, 
      author={Piotr Koszmider and Małgorzata Rojek},
      year={2025},
      eprint={2509.22376},
      archivePrefix={arXiv},
      primaryClass={math.FA},
      url={https://arxiv.org/abs/2509.22376}, 
}

@article {Peng,
    AUTHOR = {Peng, Yinhe},
     TITLE = {Distinguishing {M}artin's axiom from its restrictions},
   JOURNAL = {Adv. Math.},
  FJOURNAL = {Advances in Mathematics},
    VOLUME = {460},
      YEAR = {2025},
     PAGES = {Paper No. 110032, 51},
      ISSN = {0001-8708,1090-2082},
   MRCLASS = {03E50 (03E02 03E65)},
  MRNUMBER = {4831835},
MRREVIEWER = {C\'{e}sar\ Corral},
       DOI = {10.1016/j.aim.2024.110032},
       URL = {https://doi.org/10.1016/j.aim.2024.110032},
}

@article {AT18,
    AUTHOR = {Avil\'{e}s, Antonio and Todorcevic, Stevo},
     TITLE = {Compact spaces of the first {B}aire class that have open
              finite degree},
   JOURNAL = {J. Inst. Math. Jussieu},
  FJOURNAL = {Journal of the Institute of Mathematics of Jussieu. JIMJ.
              Journal de l'Institut de Math\'{e}matiques de Jussieu},
    VOLUME = {17},
      YEAR = {2018},
    NUMBER = {5},
     PAGES = {1173--1196},
      ISSN = {1474-7480},
   MRCLASS = {54D30 (03E15 05C05 05D10 26A21 54D55 54H05)},
  MRNUMBER = {3860391},
MRREVIEWER = {Klaas Pieter Hart},
       DOI = {10.1017/S1474748016000335},
       URL = {https://doi.org/10.1017/S1474748016000335},
}

@book {CSC,
    AUTHOR = {Cabello S\'{a}nchez, F\'{e}lix and Castillo, Jes\'{u}s M. F.},
     TITLE = {Homological methods in {B}anach space theory},
    SERIES = {Cambridge Studies in Advanced Mathematics},
    VOLUME = {203},
 PUBLISHER = {Cambridge University Press, Cambridge},
      YEAR = {2023},
     PAGES = {xi+547},
      ISBN = {978-1-108-47858-8},
   MRCLASS = {46-02 (46A16 46Mxx)},
  MRNUMBER = {4523233},
MRREVIEWER = {Sven-Ake Wegner},
}

@article {FKK13,
    AUTHOR = {Ferrer, Jes\'{u}s and Koszmider, Piotr and Kubi\'{s}, Wies\l aw},
     TITLE = {Almost disjoint families of countable sets and separable
              complementation properties},
   JOURNAL = {J. Math. Anal. Appl.},
  FJOURNAL = {Journal of Mathematical Analysis and Applications},
    VOLUME = {401},
      YEAR = {2013},
    NUMBER = {2},
     PAGES = {939--949},
      ISSN = {0022-247X},
   MRCLASS = {46B26 (46E15 54C35 54D80 54G12)},
  MRNUMBER = {3018039},
MRREVIEWER = {Ond\v{r}ej Kurka},
       DOI = {10.1016/j.jmaa.2013.01.008},
       URL = {https://doi.org/10.1016/j.jmaa.2013.01.008},
}

@article {KL21,
    AUTHOR = {Koszmider, Piotr and Laustsen, Niels Jakob},
     TITLE = {A {B}anach space induced by an almost disjoint family,
              admitting only few operators and decompositions},
   JOURNAL = {Adv. Math.},
  FJOURNAL = {Advances in Mathematics},
    VOLUME = {381},
      YEAR = {2021},
     PAGES = {Paper No. 107613, 39},
      ISSN = {0001-8708},
   MRCLASS = {47L10 (03E05 46B26 46E15 46H40 47B38 54D45 54D80)},
  MRNUMBER = {4206788},
MRREVIEWER = {Omid Zabeti},
       DOI = {10.1016/j.aim.2021.107613},
       URL = {https://doi.org/10.1016/j.aim.2021.107613},
}

@incollection {Hr14,
    AUTHOR = {Hru\v{s}\'{a}k, Michael},
     TITLE = {Almost disjoint families and topology},
 BOOKTITLE = {Recent progress in general topology. {III}},
     PAGES = {601--638},
 PUBLISHER = {Atlantis Press, Paris},
      YEAR = {2014},
   MRCLASS = {03-02 (03E05 03E35 54-02 54H05)},
  MRNUMBER = {3205494},
MRREVIEWER = {J\"{o}rg D. Brendle},
       DOI = {10.2991/978-94-6239-024-9\_14},
       URL = {https://doi.org/10.2991/978-94-6239-024-9_14},
}

@article {GHK23,
    AUTHOR = {Guzm\'{a}n, Osvaldo and Hru\v{s}\'{a}k, Michael and Koszmider, Piotr},
     TITLE = {Almost disjoint families and the geometry of nonseparable
              spheres},
   JOURNAL = {J. Funct. Anal.},
  FJOURNAL = {Journal of Functional Analysis},
    VOLUME = {285},
      YEAR = {2023},
    NUMBER = {11},
     PAGES = {Paper No. 110149, 49},
      ISSN = {0022-1236},
   MRCLASS = {46B20 (46B25 46B26)},
  MRNUMBER = {4642565},
MRREVIEWER = {Jacopo Somaglia},
       DOI = {10.1016/j.jfa.2023.110149},
       URL = {https://doi.org/10.1016/j.jfa.2023.110149},
}

@article {PS23,
    AUTHOR = {Plebanek, Grzegorz and Salguero Alarc\'{o}n, Alberto},
     TITLE = {The complemented subspace problem for {$C(K)$}-spaces: a
              counterexample},
   JOURNAL = {Adv. Math.},
  FJOURNAL = {Advances in Mathematics},
    VOLUME = {426},
      YEAR = {2023},
     PAGES = {Paper No. 109103, 20},
      ISSN = {0001-8708},
   MRCLASS = {46E15 (46B03 46B25 54G12)},
  MRNUMBER = {4592269},
MRREVIEWER = {Rigoberto Vera Mendoza},
       DOI = {10.1016/j.aim.2023.109103},
       URL = {https://doi.org/10.1016/j.aim.2023.109103},
}

@article{AvilesKrupskiLSigma,
 AUTHOR = {Avil{\'e}s, Antonio and Krupski, Miko{\l}aj},
 TITLE = {Some examples concerning ${L}{{\Sigma}}({{\leq}}{{\omega}})$ and metrizably fibered compacta},
JOURNAL={Israel J. Math.},
VOLUME = {271},
PAGES = {155–186},
 YEAR = {2025},
URL={https://doi.org/10.1007/s11856-025-2787-1},
}

@article {AMP20,
    AUTHOR = {Avil\'{e}s, Antonio and Marciszewski, Witold and Plebanek,
              Grzegorz},
     TITLE = {Twisted sums of {$c _0$} and {$C(K)$}-spaces: a solution to
              the {CCKY} problem},
   JOURNAL = {Adv. Math.},
  FJOURNAL = {Advances in Mathematics},
    VOLUME = {369},
      YEAR = {2020},
     PAGES = {107168, 31},
      ISSN = {0001-8708},
   MRCLASS = {46B25 (03E50 46B26 46E15 46M18 54D40)},
  MRNUMBER = {4091893},
MRREVIEWER = {Tanmoy Paul},
       DOI = {10.1016/j.aim.2020.107168},
       URL = {https://doi.org/10.1016/j.aim.2020.107168},
}

@article{MP18,
title = {Extension operators and twisted sums of $c_0$ and ${C}({K})$ spaces},
journal = {J.\ Funct.\ Anal.},
volume = {274},
number = {5},
pages = {1491-1529},
year = {2018},
issn = {0022-1236},
doi = {https://doi.org/10.1016/j.jfa.2017.09.004},
author = {Witold Marciszewski and Grzegorz Plebanek},
}

@article{MP09,
title = {On {B}anach spaces whose norm-open sets are ${F}_\sigma$-sets in the weak topology},
journal = {J.\ Math.\ Anal.\ Appl.},
volume = {350},
number = {2},
pages = {708-722},
year = {2009},
author = {Witold Marciszewski and Roman Pol},
}

@article {Ma89,
    AUTHOR = {Marciszewski, Witold},
     TITLE = {On a classification of pointwise compact sets of the first
              {B}aire class functions},
   JOURNAL = {Fund. Math.},
  FJOURNAL = {Polska Akademia Nauk. Fundamenta Mathematicae},
    VOLUME = {133},
      YEAR = {1989},
    NUMBER = {3},
     PAGES = {195--209},
      ISSN = {0016-2736},
   MRCLASS = {54C35 (54H05)},
  MRNUMBER = {1065902},
MRREVIEWER = {J. van Mill},
       DOI = {10.4064/fm-133-3-195-209},
       URL = {https://doi.org/10.4064/fm-133-3-195-209},
}

@article {MP12,
    AUTHOR = {Marciszewski, Witold and Pol, Roman},
     TITLE = {On {B}orel almost disjoint families},
   JOURNAL = {Monatsh. Math.},
  FJOURNAL = {Monatshefte f\"{u}r Mathematik},
    VOLUME = {168},
      YEAR = {2012},
    NUMBER = {3-4},
     PAGES = {545--562},
      ISSN = {0026-9255},
   MRCLASS = {54H05 (54C35)},
  MRNUMBER = {2993963},
MRREVIEWER = {Barbara Majcher-Iwanow},
       DOI = {10.1007/s00605-012-0433-6},
       URL = {https://doi.org/10.1007/s00605-012-0433-6},
}

@article {Po22,
    AUTHOR = {Poprawa, Wojciech},
     TITLE = {Przestrzenie zwarte generowane przez rodziny prawie roz{\l}\c{a}czne
     [Compact spaces generated by almost disjoint families]},
   
   NOTE = {Bachelor  thesis accepted by the University of Wroc{\l}aw,
              Wroc{\l}aw, Poland},
 
      YEAR = {2022},
     PAGES = {1--19},     
    }

@article{BARRIGAACOSTA20191,
title = {c-Many types of a $\Psi$-space},
journal = {Topology Appl.},
volume = {253},
pages = {1-6},
year = {2019},
author = {Hector Alonzo Barriga-Acosta and Fernando Hernández-Hernández},

}

@article{CABELLOSANCHEZ2020108571,
title = {Sailing over three problems of {K}oszmider},
journal = {J.\ Funct.\ Anal.},
volume = {279},
number = {4},
pages = {108571},
year = {2020},

author = {Félix {Cabello Sánchez} and Jesús M.F. Castillo and Witold Marciszewski and Grzegorz Plebanek and Alberto Salguero-Alarcón},

}

@article{JLoriginal,
  author = {W.B. Johnson and J. Lindenstrauss },
  year = {1974},
  title = {Some remarks on weakly compactly generated {B}anach spaces},
  journal={Israel J. Math.},
  volume={17},
  pages={219–230},
  year={1974},
}

@article{YostOriginal,
  title={The {J}ohnson-{L}indenstrauss space},
  author={David Yost},
  journal={Extracta Math.},
  volume={12},
  number={2},
  pages={185-192},
  year={1997},
}

@article{Avilésarticle,
author = {Avilés, Antonio and Martínez-Cervantes, Gonzalo and Rodríguez, José},
year = {2019},
month = {09},
pages = {},
title = {Weak⁎-sequential properties of {J}ohnson–{L}indenstrauss spaces},
volume = {276},
journal = {J.\ Funct.\ Anal.},
doi = {10.1016/j.jfa.2018.09.007}
}

@incollection{HHHsurvey,
 author = {Hern{\'a}ndez-Hern{\'a}ndez, F. and Hru{\v{s}}{\'a}k, M.},
 title = {Topology of {Mr{\'o}wka}-{Isbell} spaces},
 booktitle = {Pseudocompact topological spaces. A survey of classic and new results with open problems},
 isbn = {978-3-319-91679-8; 978-3-319-91680-4},
 pages = {253--289},
 year = {2018},
 publisher = {Cham: Springer},
 language = {English},
 doi = {10.1007/978-3-319-91680-4_8},
 zbMATH = {7121778},
 Zbl = {1447.54006}
}
\end{document}